\documentclass[twoside,11pt]{article}
\usepackage[preprint]{jmlr2e}

\usepackage{graphicx}
\usepackage{times}
\usepackage{mathptmx}
\usepackage{mathtools}
\usepackage{amsmath,amssymb}
\usepackage{bbm}
\usepackage{algorithm}
\usepackage{color}
\usepackage[title]{appendix}
\usepackage{algpseudocode}
\usepackage{caption}
\usepackage{subcaption}
\usepackage{tablefootnote}
\usepackage{hyperref}
\usepackage{comment}

\newcommand{\epsadv}{\varepsilon_{\textit{adv}}}
\newcommand{\Rb}{\mathbbm{R}}      

\newcommand{\Ac}{\mathcal{A}}
\newcommand{\Bc}{\mathcal{B}}

\newcommand{\Eb}{\mathbbm{E}}
\newcommand{\Fc}{\mathcal{F}}

\newcommand{\Lc}{\mathcal{L}}

\newcommand{\1}{\mathbbm{1}}
\newcommand{\argmax}{\mathop{\rm argmax}}
\newcommand{\argmin}{\mathop{\rm argmin}}

\newcommand{\prox}{\text{\rm prox}}

\newtheorem{assumption}{Assumption}

\newcommand\numberthis{\addtocounter{equation}{1}\tag{\theequation}}

\newcommand{\mg}[1]{{\color{blue} #1}}
\newcommand{\mgtwo}[1]{{\color{cyan} #1}}

\newcommand{\lz}[1]{{\color{red} #1}}

\renewcommand{\mg}[1]{{\color{black} #1}}
\renewcommand{\mgtwo}[1]{{\color{black} #1}}
\renewcommand{\lz}[1]{{\color{black} #1}}

\newenvironment{tightitemize}{%
    \list{{\textup{$\bullet$}}}{\settowidth\labelwidth{{\textup{\qquad}}}
    \leftmargin\labelwidth \advance\leftmargin\labelsep
    \parsep 0pt plus 1pt minus 1pt \topsep 3pt \itemsep 3pt
    }}{\endlist}

\newenvironment{tightlist}[1]{%
    \list{{\textup{(\roman{enumi})}}}{\settowidth\labelwidth{{\textup{(#1)}}}
    \leftmargin 0pt \advance\leftmargin\labelsep \itemindent \parindent
    \parsep 0pt plus 1pt minus 1pt \topsep 0pt \itemsep 0pt
    \usecounter{enumi}}}{\endlist}

       \definecolor{plum}{rgb}{0.3,0,0.7}

\jmlrheading{1}{2023}{1-24}{01/16}{--/--}{tbd}{Landi Zhu, Mert G\"urb\"uzbalaban and Andrzej Ruszczy{\'n}ski}

\ShortHeadings{Finite-Sample Guarantees for Distributionally Robust Learning}{Zhu, G\"urb\"uzbalaban and Ruszczy{\'n}ski}
\firstpageno{1}

\begin{document}

\title{Distributionally Robust Learning with Weakly Convex Losses: Convergence Rates and Finite-Sample Guarantees
}

\author{\name Landi Zhu \email lz401@scarletmail.rutgers.edu
\AND
\name Mert G\"{u}rb\"{u}zbalaban \email mg1366@rutgers.edu 
\AND
\name Andrzej Ruszczy{\'n}ski \email rusz@rutgers.edu\\
\addr Department of Management Science and Information Systems\\ Rutgers Business School, Piscataway, USA. 
      }

\editor{TBD}

\maketitle

\begin{abstract}
 We consider a distributionally robust stochastic optimization problem and formulate it as a stochastic two-level composition optimization problem with the use of the mean--semideviation risk measure. In this setting, we consider  a single time-scale algorithm, involving
two versions of the inner function value tracking: linearized tracking of a continuously differentiable loss function, and SPIDER tracking of a weakly convex \mg{loss} function.
We adopt the norm of the gradient of the Moreau envelope as our measure of stationarity and show that the sample complexity of $\mathcal{O}(\varepsilon^{-3})$ is possible in both cases,  with only the constant larger in the second case.
Finally, we demonstrate the performance of our algorithm with a robust learning example and a weakly convex, non-smooth regression example.
\end{abstract}

\section{Introduction}\label{sec-intro}
\mg{We consider distributionally robust learning problems of the form
\begin{equation}
\min_{x\in X} \max_{\mathbb{Q}\in \mathcal{M}(\mathbb{P})} \mathbb{E}_{D\sim \mathbb{Q}} \left[\ell (x,D)\right],
\label{pbm-robust-stoc-opt}
\end{equation}
where $\ell:\Rb^n\times \Rb^d\to \Rb$ is the loss function of the predictor $x$ on the random data $D$ with a perturbed distribution with probability law $\mathbb{Q}$, $\mathcal{M}(\mathbb{P})$ is a closed convex set of probability measures (the \emph{ambiguity set}) that models perturbations to the reference law $\mathbb{P}$, and $X\subset\mathbb{R}^n$ is the feasible set. Such formulations allow training predictive models from data that are robust to perturbations in the input data distribution $\mathbb{P}$, by considering the worst case of the input distribution varying in the set $\mathcal{M}(\mathbb{P})$. Such a worst-case approach to stochastic optimization; recently,
it has also become relevant to machine learning applications. Such applications include but are not limited to convex and non-convex formulations of logistic regression, deep learning, and more generally supervised learning of predictive models in a data-robust fashion with risk minimization \citep{GRZ-JOTA,zhang2021robust,zhang2022sapd+,laguel2022superquantile,mehrotra,kuhn2019wasserstein}. The challenges in these applications are that the model dimension $n$ and the number of data points may be large, and the loss functions may be both non-smooth and non-convex.}

\mg{The ambiguity set $\mathcal{M}(\mathbb{P})$ models the uncertainty about the baseline data distribution $\mathbb{P}$, and its choice and the structure of the loss function affect the computational tractability of the resulting formulations and the design of optimization algorithms for solving \eqref{pbm-robust-stoc-opt}. Various possible choices of $\mathcal{M}(\mathbb{P})$ include the Wasserstein balls around $\mathbb{P}$ \citep{robust-log-reg,esfahani2018data,mehrotra,kuhn2019wasserstein,gao2017wasserstein,sinha2018certifying}, the $f$-divergence-based uncertainty sets \citep{bagnell2005robust,duchi-namkoong,namkoong2016stochastic,zhang2021robust}, and approaches based on risk measures such as the conditional value at risk \citep{robust-cvar,laguel2022superquantile} and the mean--semideviation risk \citep{GRZ-JOTA}. In this context, the quality of a first-order optimization algorithm can be assessed in terms of its convergence rate guarantees and sample complexity, i.e., the number of data points to be sampled for finding an approximate first-order stationary solution.}

We start with introducing the mean--semideviation based modeling of the uncertainty set $\mathcal{M}(\mathbb{P})$ before summarizing our contributions. 
For each model parameter $x\in X$, we consider the random loss $Z_x = \ell(x,D)$, defined on a sample space $\Omega$ equipped with a $\sigma$-algebra $\mathcal{F}$. We assume that the expectation $\mathbb{E}(Z_x) = \int_\Omega Z_x(\omega)\;\mathbb{P}(d\omega)$ is finite; i.e $Z_x \in \mathcal{L}_1(\Omega,\mathcal{F},\mathbb{P})$. We can evaluate its quality by the mean--semideviation risk measure defined as
\begin{equation}
\label{msd}
\rho[Z_x] =  \Eb[Z_x]  + \varkappa\, \Eb\Big[\max\big(0,Z_x - \Eb[Z_x]\big)\Big], \quad \varkappa \in [0,1],
\end{equation}
which penalizes the expected value of the random losses for large values of losses that exceed the expected value. This risk measure is a coherent risk measure enjoying several desirable properties and is used in many contexts in statistics and stochastic optimization \citep{rockafellar2006generalized,GRZ-JOTA}. It is well-known that the mean--semideviation risk $\rho[Z_x]$ of the random variable $Z_x$, admits the following dual representation \citep{RuSh:2006a}:
\begin{eqnarray}
\rho[Z_x] &=& \max_{\mu \in \Ac} \int_\varOmega Z_x(\omega)\mu(\omega)\;\mathbb{P}(d\omega) \nonumber\\
&=&
{\max_{\mathbb{Q}~:~ \frac{d\mathbb{Q}}{d\mathbb{P}}\in \Ac} \int_\varOmega Z_x(\omega)\;\mathbb{Q}(d\omega)}=
\max_{{\mathbb{Q} ~:~ \frac{d\mathbb{Q}}{d\mathbb{P}}\in \Ac}}\Eb_\mathbb{{Q}}[Z_x],\label{eq-dual-representation}
\end{eqnarray}
where $$\Ac = \big\{ \mu= \1+ \xi - \Eb[\xi]: \ \xi\in \Lc_{\infty}(\varOmega,\Fc,\mathbb{P}), \  \|\xi\|_\infty \le \varkappa,\ \xi \ge 0 \big\}$$ is a convex and closed set. Thus, from \eqref{msd} and \eqref{eq-dual-representation}, the min--max form
\eqref{pbm-robust-stoc-opt} with the ambiguity set
\begin{equation}
 \mathcal{M}(\mathbb{P}) = \Big\{ \mathbb{Q} : \frac{d\mathbb{Q}}{d\mathbb{P}} \in \mathcal{A}\Big\}
\label{eq-Mp}
\end{equation}
is equivalent to 
\begin{equation}
\min_{x\in X} F(x) = \min_{x\in X}  \rho[\ell(x,D)]  = 
\min_{x\in X}\; f(x,h(x)) \label{main_prob}
\end{equation}
with the functions
\begin{align}
f(x,u) &= \Eb\Big[u +  \varkappa\, \max\big(0,\ell(x,D) - u\big)\Big],\label{outer}\\
h(x) &= \Eb[\ell(x,D)].\nonumber 
\end{align}
In this way, using the mean-semideviation risk measure, one can convert the min-max problem into a two-level stochastic optimization problem which achieves an \emph{implicit} robust formulation, with the level of robustness controlled by the parameter $\varkappa$: for $\varkappa=0$ the uncertainty set \eqref{eq-Mp} contains
only the original probability measure $\mathbb{P}$, while for $\varkappa>0$ the measures $\mathbb{Q}\in \mathcal{M}(\mathbb{P})$ are distortions of $\mathbb{P}$. The range of the relative distortions allowed, $\frac{d\mathbb{Q}}{d\mathbb{P}} - \1$,
is controlled by $\varkappa$. 
The challenge is that this formulation is non-smooth, and typically non-convex when the loss is non-convex. Furthermore,
the use of the expected value inside the nonlinear function $f(x,\cdot)$ results in a bias of stochastic subgradient estimates.

\textbf{Contributions.} In this paper, we develop an optimization algorithm called \emph{stochastic compositional subgradient} (SCS) for solving \eqref{pbm-robust-stoc-opt} based on the reformulation \eqref{main_prob}, and establish finite-time convergence analysis and sample complexity results. Our algorithm is a variation and a simplified version of the single-time scale method proposed in \cite{GRZ-JOTA}, where the subgradients are no longer averaged and subgradient tracking is no longer needed. In  sections \ref{s:2} and \ref{s:3}, we assume a continuously differentiable loss function, and build on a projected subgradient descent framework, use linearized tracking, and employ the gradient of the Moreau envelope as our convergence metric. We prove  that the SCS method has a sample complexity of order $\mathcal{O}(\varepsilon^{-3})$ in this case.
In section  \ref{s:4} we assume a $\delta$-weakly convex loss function (see definition \eqref{weak_conv})  and use the SPIDER estimator \mgtwo{from \cite{fang2018spider}} to estimate the expectation of the losses. It is worth stressing that both $f(\cdot,\cdot)$ and $h(\cdot)$ in \eqref{main_prob} are nonsmooth and $h(\cdot)$ is nonconvex in this case. We prove that the SCS method has the same sample complexity of order $\mathcal{O}(\varepsilon^{-3})$, with a larger constant, though.

\paragraph{Related Work}
\mg{The convexity and non-smoothness structure of the (two-level) stochastic composite optimization problem \eqref{main_prob} is determined by the choice of the loss. When the loss is convex and possibly non-smooth, the composite objective $F(x)$ will also be convex and non-smooth in $x$. In this case, multi-level convex stochastic optimization algorithms such as \cite{wang2017stochastic} will be applicable, implying a sample complexity of $\mathcal{O}(\varepsilon^{-4})$ when the loss is convex and $\mathcal{O}(\varepsilon^{-1.5})$ when the loss is strongly convex. }
\lz{
When the loss is non-convex, irrespective of whether it is smooth or not, the composite objective $F(x)$ will be non-convex and non-smooth in $x$. The convergence rate for this general setting is not available, but if we only consider a smooth problem, there exist some complexity results. 
} 
In \cite{wang2017stochastic}, the authors analyzed stochastic gradient algorithms with different assumptions on the objective, and prove sample complexities $\mathcal{O}(\varepsilon^{-3.5})$ and $\mathcal{O}(\varepsilon^{-1.25})$ for smooth convex problems and smooth strongly convex problems, respectively. These rates can be further improved with proper regularization \citep{wang2016accelerating}. In \cite{ghadimi2020single}, the authors propose a single time-scale Nested Averaged Stochastic Approximation (NASA) method for smooth nonconvex \mgtwo{composition optimization problems} and  prove the sample complexity of $\mathcal{O}(\varepsilon^{-2})$. For higher-level (more than two) problems, \cite{ruszczynski2021stochastic} establishes asymptotic convergence of a stochastic subgradient method by analyzing a system of differential inclusions, along with a sample complexity of $\mathcal{O}(\varepsilon^{-2})$ when smoothness is assumed. Another level-independent rate of $\mathcal{O}(\varepsilon^{-2})$ is obtained for smooth multi-level problems in \cite{balasubramanian2022stochastic} without the boundedness assumption. 

There are also approaches dealing with \eqref{pbm-robust-stoc-opt} not based on composite stochastic optimization. In particular,  \cite{ho2022adversarial} considered linear classification problems subject to Wasserstein ambiguity sets for the ``zero-one loss" which is non-convex and non-smooth. The authors showed that this problem is equivalent to minimizing a regularized ramp loss objective and proposed a class of
smooth approximations to the ramp loss, where smooth problems can be solved (approximately) with standard continuous optimization algorithms. There are also other approaches which can provide complexity results when the loss is either smooth or convex.
\lz{
\paragraph{Smooth losses}
\cite{sinha2018certifying} formulate $\mathcal{M}(\mathbb{P})$ as a $\rho$-neighborhood of the probability law $\mathbb{P}$ under the Wasserstein metric. They show that for a smooth loss and small enough robustness level $\rho$, the stochastic gradient descent (SGD) method can achieve the same rate of convergence as that in the standard smooth non-convex optimization. In \cite{jin2021non}, the authors
consider smooth and Lipschitz non-convex losses  
and use
a soft penalty term based on $f$-divergence to model the distribution shifts. They analyzed the mini-batch
normalized SGD with momentum and proved an $\mathcal{O}(\varepsilon^{-4})$ sample complexity \mgtwo{(for the norm of the gradient of the loss to be at most $\varepsilon$) which also matches the rates that can be obtained in standard smooth non-convex optimization. For a smoothed version of the CVaR, the authors obtain similar convergence guarantees for smooth non-convex losses that are Lipschitz.}}
\mgtwo{In \cite{soma2020statistical}, the authors proposed a conditional value-at-risk (CVaR) formulation. They show that for convex, Lipschitz and smooth losses their SGD-based algorithm has a complexity of $\mathcal{O}(1/\varepsilon^2)$, whereas for non-convex, smooth and Lipschitz losses, the authors obtain a complexity of $\mathcal{O}(1/\varepsilon^6)$.
In \cite{curi2020adaptive},
the authors proposed an adaptive sampling algorithm for  stochastically optimizing the CVaR of the empirical distribution of the loss, and reformulated this optimization problem as a two-player game based on the dual representation of CVaR. For convex problems, they obtain a regret bound of $\mathcal{O}(T)$ over $T$ iterations, and for non-convex problems, they obtain a regret bound of $\mathcal{O}(T)$ assuming access to an inexact empirical risk minimization (ERM) oracle. However, implementing this oracle for non-convex problems requires solving a weighted empirical loss minimization in every iteration and this is NP-hard in general \cite[Sec. 4.2]{curi2020adaptive}. 

\mgtwo{When the loss is smooth and Lipschitz continuous on the primal space $X$, sample-based approximations of \eqref{pbm-robust-stoc-opt} where the expectation is approximated by a finite average taken over the data points results 
in smooth non-convex/merely concave min-max optimization problems where the dual space $\mathcal{M}(P)$ is finite-dimensional when determined by $f$-divergences \citep{zhang2022sapd+}. In this case, primal-dual algorithms are applicable and the SAPD+ algorithm from \citep{zhang2022sapd+} provides an $\mathcal{O}(1/\varepsilon^6)$ complexity.} There are also primal approaches. In particular, \cite{tianbao-KL} propose a primal method for solving a class of distributionally robust optimization (DRO) problems with smooth non-convex objectives in the sample-based approximation form. They consider a KL divergence regularization on the dual variable and convert the DRO problem into a smooth two-level compositional problem. For smooth losses, the authors show an $\tilde{\mathcal{O}}(1/\varepsilon^{3/2})$ complexity for computing an approximate solution $w_\varepsilon$, i.e., a solution satisfying $\|\nabla F_p(w_\varepsilon)\|^2\leq \varepsilon$ where $F_p$ is the primal function. Under the additional PL condition, the authors also derive an improved complexity of  $\mathcal{O}(\frac{1}{\varepsilon})$ for smooth losses when the suboptimality of the primal function is the performance metric. More recently, \cite{qi2022stochastic} considered the KL-constrained DRO for smooth losses and proposed scalable compositional algorithms that can work with a constant batch size at every iteration. Applications of sample-based DRO formulations to various problems in machine learning, such as handling imbalanced data \citep{qi2022attentionalbiased} and partial AUC maximization \citep{zhu2022auc}, have also been studied. \mgtwo{Our formulation does not require sample-based approximations and can handle the general case when $\mathcal{M}(\mathbb{P})$ may be infinite-dimensional}; furthermore, we obtain complexity results for weakly convex losses that are not only non-convex but at the same time that can be non-smooth. \looseness=-1
\paragraph{Convex losses}
If formulated as finite-dimensional convex programs \citep{robust-log-reg,esfahani2018data,mehrotra,kuhn2019wasserstein}, the distributionally robust problem \eqref{pbm-robust-stoc-opt} can be solved in polynomial time. When $\mathcal{M}(\mathbb{P})$ is defined via the $f$-divergences and the loss is convex and smooth, a sample-based approximation of \eqref{pbm-robust-stoc-opt} can be solved with a bandit mirror descent algorithm \citep{namkoong2016stochastic} with the number of iterations comparable to that of the SGD. For convex losses in the same formulation, conic interior point solvers or gradient descent with backtracking Armijo line-searches \citep{duchi-namkoong} can also be used but this can be computationally expensive for some applications when the dimension or the number of samples is large. When the uncertainty set $\mathcal{M}(\mathbb{P})$ is based on the empirical distribution of the data and is defined via the $\chi^2$-divergence or CVaR, and the loss is convex and Lipschitz, \cite{levy2020large} proposed algorithms that achieve an optimal $\mathcal{O}(\varepsilon^{-2})$ rate which is independent of the training dataset size and the number of parameters. 

When $\ell(\cdot,D)$ is non-convex and non-differentiable, distributionally robust stochastic optimization problems lead to non-convex non-smooth min-max optimization problems. To our knowledge, in this general case, none of the existing algorithms admit provable convergence guarantees to a stationary point of \eqref{pbm-robust-stoc-opt} and do not admit iteration complexity bounds. \mgtwo{Our results apply to this setting and provide iteration complexity estimates for weakly convex losses that may be non-smooth.}
}

\textbf{Notation and Preliminaries.} \mg{A function $q: \mathbb{R}^n \to \mathbb{R}$ is called $\delta$-weakly convex, if the regularized function $x\mapsto q(x)+\frac{\delta}{2}\|x\|^2$ is  convex  \citep{nurminskii1973quasigradient}}. This is a broad class of functions that can be non-smooth and non-convex, including all convex functions and smooth functions with a globally Lipschitz continuous gradient.
A $\delta$-weakly convex function $q(x)$ has also the following property: at every point $x\in \Rb^n$ a vector $g\in \Rb^n$ exists such that
\begin{equation}
\mg{q(y) \ge q(x) + \langle g,y-x\rangle -\frac{\delta}{2}\|y-x\|^2,\qquad \forall y \in \Rb^n,}
\label{weak_conv} 
\end{equation}
\mg{(see e.g. \cite{davis2018subgradient})}. The set $\partial q(x)$ of vectors $g$ satisfying the above relation is the \emph{subdifferential} of $q(\cdot)$ at $x$; it is nonempty, convex, and closed. In fact, it coincides with the Clarke subdifferential for this class of functions \citep{Rockafellar-Wets}. \mg{We say that a continuously differentiable function $q:\mathbb{R}^n\to\mathbb{R}$ is $L$-smooth on a convex set $X$, if $\nabla q(x)$ is Lipschitz continuous on $X$, i.e. $\|\nabla q(x) - \nabla q(y)\| \leq L\|x-y\|$ for all $x,y\in X$.} 
\section{A Stochastic Compositional Subgradient (SCS) Method}
\label{s:2}

\mg{We first consider the case when the loss function is continuously differentiable, the non-smooth case will be addressed later in Section {\color{plum} \ref{s:4}}. } 

\begin{assumption}
The set $X\subset \Rb^n$ is convex and compact.
\label{as1}
\end{assumption}

\begin{assumption}
 For all $x$ in a neighborhood of the set $X$:
\begin{tightlist}{(ii)}
\item The function $\ell(x,\cdot)$ is integrable;
\item The function $\ell(\cdot,D)$ is continuously differentiable
and integrable constants $\tilde{\Delta}_h(D)$ and $\Tilde{\delta}(D)$ exist
such that 
\[
\| \nabla \ell(x,D)\| \le \tilde{\Delta}_h(D), \quad \forall\, D\in \mgtwo{\Rb^d},
\]
and 
\[
\| \nabla \ell(x,D)- \nabla \ell(y,D)\|\le \Tilde{\delta}(D)\|x - y\|, \quad \forall\, 
\mgtwo{x,y}\in X,\quad \forall\,D\in \mgtwo{\Rb^d}.
\]
\end{tightlist}

\label{as2}
\end{assumption}
\begin{remark}
\label{remark-weakly-convex}
\lz{Since the loss function $\ell(x, D)$ is $\Tilde{\delta}(D)$-smooth and the feasible set $X$ is compact, $\ell(x, D)$ is also
$\Tilde{\delta}(D)$-weakly convex (\mgtwo{see the definition of the weak convexity in the last paragraph of Section \ref{sec-intro}}). \mg{Assumption \ref{as1} and \ref{as2} guarantees that the expected value function $h(\cdot)$ is well defined and $\delta$-smooth, \lz{$\delta$-weakly convex} on $X$, with $\delta = \Eb[\Tilde{\delta}(D)]$.} } 
\end{remark}

\mg{Assumptions \ref{as1} and \ref{as2} are satisfied for many problems in statistical learning including non-convex constrained formulations of various classification and regression tasks such as deep learning, least squares and logistic regression \citep{GRZ-JOTA,pmlr-v119-negiar20a}.} 

\mg{If a subgradient of $f(\cdot)$ and the gradient of $h(\cdot)$ were known, a subgradient of the composite function $F(x) = f(x,h(x))$ could be calculated by an application of the chain rule \citep{Rockafellar-Wets}, i.e. if $\begin{bmatrix} g_{fx}\\ g_{fu }\end{bmatrix} \in \partial f(x,u)$ and $g_h=\nabla h(x)$, then we would have
 \begin{equation} g_{fx} + g_{fu} g_h \in \partial F(x).
 \label{eqn-chain-rule}
 \end{equation}}
\mg{Unfortunately, in our setting, we neither have access to the subgradients in \eqref{eqn-chain-rule} nor to the value of $h(x)$; we can only obtain their stochastic estimates.}
\mg{To address this, our proposed method, stochastic compositional subgradient (SCS) generates approximate solutions $\big\{x^k\big\}_{k=1,2,\dots}$ in $\Rb^n$ based on a projected stochastic subgradient update rule that estimates the subgradient of the composite function $F(x)$ (by relying on the stochastic subgradients of $f$ and $h$) and projects the iterates back to the constraint set $X$ which ensures that the iterates stay bounded. Our method described in Algorithm \ref{alg1} also generates  
random inner function estimates $\big\{u^k\big\}_{k=1,2,\dots}$} in $\Rb$, 
\mg{where we assume access to unbiased stochastic estimates of the subgradients of $f$ and $h$ and values of $h$ with a bounded variance. More precisely, denoting by$\Fc_k$ the $\sigma$-algebra generated by $\{\mg{x^0, u^0}, x^1,u^1,\dots,x^k,u^k\}$ where \mg{$x^0 \in X$ and $u^0 \in \mathbb{R}$ are the initializations}, we make the following assumption.}
\begin{assumption}
 For all $k$, we have access to random vectors $\tilde{g}_f^k$, $\tilde{g}_h^k$, $\tilde{J}^k$, and random variables $\tilde{h}^k$ satisfying the conditions:
\begin{tightitemize}
\item
$\tilde{g}_f^k = g_f^k+e_f^k,\quad g_f^k\in \partial f(x^k,u^k),\quad \Eb\big[e_f^k\big|\Fc_k\big] =0,\quad \Eb\big[\|e_f^k\|^2\big|\Fc_k\big] \le \sigma^2$;
\item
$\tilde{g}_h^k = g_h^k+e_h^k,\quad g_h^k= \nabla h(x^k),\quad \Eb\big[e_h^k\big|\Fc_k\big] =0,\quad \Eb\big[\|e_h^k\|^2\big|\Fc_k\big] \le \sigma^2$;
\item
$\tilde{h}^k = h(x^k)+e_\ell^k,\quad \Eb\big[e_\ell^k\big|\Fc_k\big] =0,\quad \Eb\big[\|e_\ell^k\|^2\big|\Fc_k\big] \le \sigma^2$;
\item $\tilde{J}^{\,k} = g_h^{k} + E^{k}, \quad
\Eb\big\{E^{k}\big|\Fc_k\big\} = 0,\quad \Eb\big\{\|E^{k}\|^2|\Fc_k\big\}\le \sigma^2;$
\end{tightitemize}
 where $\sigma$ is a constant,
and the errors $e_f^k$, $e_h^k$, $e_\ell^k$, and $E^{k}$ are conditionally independent, given $\Fc_k$.
\label{as3}
\end{assumption}
\lz{
\begin{remark}
Under Assumptions \ref{as1} and \ref{as2}, all the values and subgradients of $f$ and $h$:  $g^k_f$, $g^k_h$ and $h(x^k)$, are bounded. We denote for all $x\in X$, $u\in \mathbb{R}$, 
\[
\|\partial_x f(x,u)\|\le \Delta_{fx},\qquad \|\nabla h(x)\|\le \Delta_{h},
\]
with $\Delta_{h} = \Eb[ \tilde{\Delta}_h(D)]$.
Under Assumption \ref{as3},  the following stochastic estimate of an element of $\partial F(x^k)$ resulting
from replacing the true (sub)gradients in \eqref{eqn-chain-rule} with their random estimates has a bounded expected square norm:
\begin{equation}
    \Eb\big[\|\tilde{g}_{fx}^k + \tilde{g}_{fu}^k \tilde{g}_h^k\|^2 \,\big|\,\Fc_k\big] \le M^2 \quad \mbox{with} \quad M^2 = (\Delta_{fx}+\Delta_{h})^2+2\sigma^2+2\sigma\Delta_h, 
    \label{gF_bound}
\end{equation}
\mgtwo{where we used conditional independence of the random errors, Cauchy-Schwarz inequality and the fact that $\tilde{g}_{fu}^k \in [0,1]$ implied by \eqref{def_gu} }.
\end{remark}
}
\mg{A common setting in statistical learning and stochastic optimization is to estimate the subgradients based on randomly sampled subsets of data points with replacement \citep{bottou2010large}. In this setting, when the domain $X$ is unbounded, it is possible that the variance of such stochastic subgradient estimator can be unbounded \citep{gurbuzbalaban2021heavy,jain2018accelerating,decentralized-heavy}. However, in our setting, the feasible set $X$ is compact, therefore Assumption \ref{as3} will be naturally satisfied (see Section \ref{subsec-assump-hold})}. 

\begin{algorithm}
\caption{\mgtwo{SCS} method}\label{alg1}

\begin{algorithmic}[1]
\small
\Require initial point $x^0 \in X$, $u^0 \in \Rb$, a constant stepsize $\tau \in \big(0,1\big]$.
\For{$k=0,1,...,N-1$}
\State 
Sample $D_1^{k+1}$, $D_2^{k+1}$ and $D_3^{k+1}$ conditionally \mgtwo{independently} on $\mathcal{F}_k$
and obtain the estimates
\begin{align}
G^{k} &\in \partial_x \ell(x^{k}, D_1^{k+1}), \label{def_G}\\
\tilde{g}_{fx}^{k} &= \begin{cases} 0 & \text{if } \ell(x^{k}, D_1^{k+1}) < u^{k},\\
                                    \varkappa G^{k} & \text{if } \ell(x^{k}, D_1^{k+1}) \ge u^{k},
                                    \end{cases}\label{def_gx}\\
\tilde{g}_{fu}^{k} &=\begin{cases} 1 & \text{if } \ell(x^{k}, D_1^{k+1}) < u^{k},\\
                                   1-\varkappa & \text{if } \ell(x^{k}, D_1^{k+1}) \ge u^{k},
                                    \end{cases}\label{def_gu}\\
\tilde{g}_{h}^{k} &\in \partial_x \ell(x^k, D_2^{k+1}),\\
\tilde{J}^{k} &\in \partial_x \ell(x^k, D_3^{k+1}),\\
\tilde{h}^{k} &= \frac{1}{3}(\ell(x^k, D_1^{k+1})+\ell(x^k, D_2^{k+1})+\ell(x^k, D_3^{k+1})).\label{def_h}
\end{align}
\State Update the solution estimate
\begin{equation*}
x^{k+1} = \Pi_X\Big( x^k - \mgtwo{\tau} \big( \tilde{g}_{fx}^k + \tilde{g}_{fu}^k \tilde{g}_h^k\big)^T\Big),
\end{equation*}
\State Update the inner function estimate
\begin{equation}
u^{k+1} = u^k + \mgtwo{\tau} \big( \tilde{h}^k - u^k)+ \tilde{J}^k\big(x^{k+1}-x^k\big).\label{update_u}
\end{equation}
\EndFor
\Ensure $x^R$ with $R$ uniformly sampled from $\{0,1,\dots,N-1\}$.
\end{algorithmic}
\end{algorithm}
\section{Convergence rate for continuously differentiable losses}
\label{s:3}

Because of the formulation of the mean-semideviation risk measure involving a non-smooth $\max(\cdot)$ term, the outer function \eqref{outer} may be nonsmooth even when the loss function  $\ell(x,D)$ is continuously differentiable.
Therefore, \mg{a challenge is that the  problem \eqref{main_prob} is non-smooth and non-convex}. \mg{However, we will argue that it is weakly convex.
For weakly convex objectives, a popular metric for determining first-order stationarity is the norm of the gradient of the Moreau envelope \citep{Moreau1965ProximitED} which we will introduce next.} We first 
consider an alternative formulation of the main problem (\ref{main_prob}):
\begin{equation*}
\min_{x\in \mathbbm{R}^n}\; \varphi(x) := F(x) + r(x),
\end{equation*}
where $F(x) = f(x,h(x))$ and $r(x)$ is the indicator function of the convex and compact feasible set  $X\subset \Rb^n$, i.e. $r(x) = 0$ if $x\in X$ and $r(x)=+\infty$ otherwise.
The Moreau envelope and the proximal map are defined as 
\begin{align*}
    \varphi_\lambda(x) &:= \min_{y} \{ \varphi(y) + \frac{1}{2\lambda}\|y-x\|^2 \},\\
    \prox_{\lambda\varphi}(x) &:= \argmin_y \{ \varphi(y) + \frac{1}{2\lambda}\|y-x\|^2 \},
\end{align*}
respectively. 
\lz{Since the inner function $h(x)$ is $\delta$-weakly convex 
 and the outer function $f(x,u)$ is weakly convex with respect to $x$ and convex and nondecreasing with respect to $u$ (see Remark \ref{remark-weakly-convex})}, the composite function $F(x)$ is also $\rho$-weakly convex with $\rho = (1+2\varkappa)\delta$. In this case $\varphi_\lambda(x)$ is \mg{continuously differentiable} for $\lambda\in (0, \rho^{-1})$ \citep{Moreau1965ProximitED} \mg{with the gradient}
\begin{equation}
    \nabla \varphi_\lambda(x) = \lambda^{-1}(x-\prox_{\lambda\varphi}(x))\label{me_grad}.
\end{equation}
It can also be shown that the quantity $\|\nabla \varphi_\lambda(x)\|$ is a measure of stationarity, i.e. when $\|\nabla \varphi_\lambda(x)\|$ is small, $x$ will be close to some \textit{nearly stationary point} $\hat{x}$, which in turn, has the subdifferential close to 0 \citep{davis2019stochastic}, \mg{i.e. $\hat{x}$ satisfies the following relations:} 
\lz{
\[
\left\{
\begin{aligned}
  \|\hat{x}-x\| &= \lambda \|\nabla \varphi_\lambda(x)\|,\\
  \varphi(\hat{x})&\le\varphi(x),\\
  \mbox{dist}(0;\partial\varphi(\hat{x}))&\le \|\nabla \varphi_\lambda(x)\|.
\end{aligned} 
\right.
\]
}
\mg{Here, $\mbox{dist}(0;\partial\varphi(\hat{x}))$ denotes the distance of the origin to the set $\partial\varphi(\hat{x})$}. \mg{Therefore, the convergence guarantees for the gradient of the Moreau envelope in this paper, can be converted to guarantees in terms of the subdifferential.}

Now we can proceed to prove the convergence rate of the SCS method for a continuously differentiable loss function. \mg{First, we quantify how well the inner  function estimates $\{u^k\}$ track the  sequence $\{h(x^k)\}$.}
\begin{lemma}
\label{smooth_error}
If Assumptions \ref{as1}, \ref{as2}, and \ref{as3} hold, the sequence $\{u^k\}$ generated by Algorithmn \ref{alg1} satisfies:
\begin{align}
\Eb\big[ |u^{k} - h(x^{k})| \big] \le  \sigma(1+M)\tau^{1/2} + {\delta} M \tau &\lz{+(1-\tau)^{k}|u^0-h(x^0)|},\nonumber \\ &\qquad k = 0,1,...,N-1.
\label{tracking_err}
\end{align}
\end{lemma}
\begin{proof}
Under Assumptions \ref{as1} and \ref{as2}, the inner function $h(x)$ is $\delta$-smooth.  Therefore,
\[
h(x^{k+1}) = h(x^k) + \big[g_h^{k}\big]^T(x^{k+1}-x^k) + A_k, \quad  \| A_k \|\le {\delta}\| x^{k+1}-x^k\|^2.
\]
From the update rule \eqref{update_u} for $\{u^k\}$, we have 
\begin{equation*}
u^{k+1} = u^k + \tau \big( {h}(x^k)- u^k) + \tau e^k_\ell + \big[g_h^{k}\big]^T\big(x^{k+1}-x^k\big) + E^{k}(x^{k+1}-x^k) .
\end{equation*}
Thus,
\[
u^{k+1} - h(x^{k+1}) = (1-\tau)\big[ u^k - h(x^k)\big] + \tau e^k_\ell + E^{k}(x^{k+1}-x^k) - A_k.
\]
\mg{By using this equality recursively}, \lz{we obtain}
\begin{multline}
\label{dynsys}
u^{k+1} - h(x^{k+1}) =  \\\sum_{j=0}^k (1-\tau)^{k-j}\big( \tau e^j_\ell + E^{j}(x^{j+1}-x^j) - A_j\big)
\lz{+(1-\tau)^{k+1}(u^0-h(x^0))}.
\end{multline}
The norms of the martingale terms can be easily bounded:
\[
\Eb\bigg[ \Big(\sum_{j=0}^k (1-\tau)^{k-j}\tau e_\ell^j\Big)^2 \bigg] \le \sum_{j=0}^k (1-\tau)^{2(k-j)} \sigma^2\tau^2
\le \frac{\sigma^2\tau^2 }{1- (1-\tau)^2} \le \sigma^2\tau,
\]
and thus
\begin{equation}
\label{martingale1}
\Eb\bigg[ \Big|\sum_{j=0}^k (1-\tau)^{k-j}\tau e_\ell^j\Big| \bigg] \le \sigma\tau^{1/2}.
\end{equation}
Observe that by the conditional independence of $E^k$ and $x^{k+1}-x^k$, we have $\Eb\big[  E^{k}(x^{k+1}-x^k)\big|\Fc_k\big]=0$. Furthermore,
by the Cauchy--Schwartz inequality and the non-expansiveness of the projection operator
\[
\Eb \Big[ \big| E^{k}(x^{k+1}-x^k)\big|^2\,\Big| \,\Fc_k \Big]\le  \sigma^2 M^2 \tau^2,
\]
where $M$ is the constant from \eqref{gF_bound}.  Thus, similar to \eqref{martingale1}, we obtain
\begin{equation}
\label{martingale2}
\Eb\bigg[ \Big|\sum_{j=0}^k (1-\tau)^{k-j}E^{j}(x^{j+1}-x^j)\Big| \bigg] \le 
\sigma M \tau^{1/2}.
\end{equation}
The third sum can be bounded directly:
\begin{equation}
\label{sum3}
\Eb\bigg[ \Big|\sum_{j=0}^k (1-\tau)^{k-j}A_j\Big| \bigg] \le {\delta} M \tau^2 \Eb\bigg[ \Big|\sum_{j=0}^k (1-\tau)^{k-j}\Big| \bigg] \le \delta M \tau. 
\end{equation}
Plugging the estimates \eqref{martingale1}--\eqref{sum3} into \eqref{dynsys} we conclude that
\[
\Eb \big|u^{k+1} - h(x^{k+1}) \big| \le \sigma(1+M)\tau^{1/2} + {\delta} M \tau
\lz{+(1-\tau)^{k+1}|u^0-h(x^0)|},
\]
as required.
\end{proof}


We now consider the Moreau envelope $\varphi_{1/\bar{\rho}}(x)$ with $\bar{\rho} = \rho+(1+\varkappa)\delta$, 
\mg{and obtain a bound for the expected squared norm of its gradient at an iterate $x^R$ that is randomly chosen among the first $N$ iterates. Complexity results for unconstrained stochastic weakly convex minimization exist \citep{davis2018stochastic} but such results are not directly applicable to our setting, because we have a two-level non-smooth weakly convex problem in $x$. Our proof leverages the monotonicity of $f(x,u)$ with respect to $u$ to handle its non-smoothness while exploiting the weak convexity of the loss with respect to $x$.}

\begin{theorem}
\label{smooth_rate}
Suppose Assumptions \ref{as1} to \ref{as3} hold. \mg{For any given iteration budget $N$, consider the trajectory $\{x^k\}_{k=0}^{N-1}$ of Algorithm \ref{alg1}}. We have
\begin{equation*}
    \Eb[\|\nabla \varphi_{1/\bar{\rho}}(x^R)\|^2]
    \le 2\frac{\varphi_{1/\bar{\rho}}(x^{0})-\lz{\min_{x\in X} F(x)} \lz{ + 2\bar{\rho}|u^0-h(x^0)|}+NC_3\tau^{3/2}}{N\tau},
\end{equation*}
where \mg{$\bar{\rho} = \rho+(1+\varkappa)\delta$}, 
\lz{$C_3 = 2\bar{\rho}\sigma(1+M)$,} the expectation is taken with respect to the trajectory generated by Algorithm \ref{alg1} and the random variable $R$ that is uniformly sampled from 
$\{0,1,...,\mg{N}-1\}$ independently of the trajectory. 
\end{theorem}
\begin{proof}
Defining $\hat{x}^k := \prox_{\varphi/\bar{\rho}}(x^k)$, we have
\begin{multline}
 f(\hat{x}^k,h(\hat{x}^k))-f(x^k,u^k) 
    = h(\hat{x}^k)-u^k \\{} + \varkappa \Eb\big[\max(0,\ell (\hat{x}^k,D_1^{k+1})-h(\hat{x}^k)) 
   - \max(0,\ell (x^k,D_1^{k+1})-u^k)\big|\Fc_k\big]. \label{eqn-diff-of-max}
\end{multline}
\mg{On the other hand,} according to Algorithm \ref{alg1}, if we denote
\begin{equation*}
    I^k = \begin{cases} 1 & \mbox{if} \quad \ell(x^{k}, D_1^{k+1}) \ge u^{k}, \\
0 & \mbox{if} \quad \ell(x^{k}, D_1^{k+1}) < u^{k},
\end{cases}
\end{equation*}
we can write $$\tilde{g}_{fx}^{k} = \varkappa I^k G^{k} \quad  \mbox{and} \quad \tilde{g}_{fu}^{k} = 1 - \varkappa I^k,$$ 
where we used the definitions \eqref{def_G}, \eqref{def_gx}, \eqref{def_gu}.
\mg{We can also estimate the difference of the ``max" terms in \eqref{eqn-diff-of-max} as}
\begin{align}
    \max&(0,\ell (\hat{x}^k,D_1^{k+1})-h(\hat{x}^k)) - \max(0,\ell (x^k,D_1^{k+1})-u^k) \nonumber\\
    &\ge
    \begin{cases}
    \ell (\hat{x}^k,D_1^{k+1})-h(\hat{x}^k) - (\ell (x^k,D_1^{k+1})-u^k),& \qquad \text{if } \ell (x^{k},D_1^{k+1}) \ge u^{k}, \nonumber\\
     0, &\qquad \text{if }\ell (x^{k},D_1^{k+1}) < u^{k},\nonumber\\
    \end{cases}\nonumber\\
    &= I^{k}\big(\ell (\hat{x}^k,D_1^{k+1}) - \ell(x^{k},D_1^{k+1}) - (h(\hat{x}^k) - u^k)\big)\nonumber\\
    &\ge I^{k}\big(\langle G^{k}, \hat{x}^k-x^k\rangle - \frac{\lz{\tilde{\delta}(D_1^{k+1})}}{2}\|\hat{x}^k-x^k\|^2 - (h(\hat{x}^k) - u^k)\big),
    \label{eqn-diff-of-max-2}
\end{align}
\mg{where we used the definition \eqref{def_G} of $G^k$ and the inequality \eqref{weak_conv} with $q(x)=\ell(x,D_1^{k+1})$ due to the weak convexity of $\ell$.} 
Denoting $$A^k = \varkappa\Eb\big[I^{k}G^{k}\big|\Fc_k\big],$$ 
\mg{and using \eqref{eqn-diff-of-max} and \eqref{eqn-diff-of-max-2}},
we obtain the following lower bound 
\begin{align}
    f(\hat{x}^k,h(\hat{x}^k))-f(x^k,u^k) &\ge \langle \hat{x}^k-x^k,A^k\rangle-\frac{\varkappa\delta}{2}\|\hat{x}^k-x^k\|^2\nonumber \\
    &\qquad +(1-\varkappa\Eb\big[I^{k}\big|\Fc_k\big]) (h(\hat{x}^k)-u^k). \label{ineq-lower-bound-f-diff}
\end{align}
Denoting $B^k=(1-\varkappa\Eb\big[I^{k}\big|\Fc_k\big])g_h^k$, we can estimate the last term as
\begin{equation}
\label{ineq-lower-bound-f-diff-last-term}
\begin{aligned}
    \lefteqn{(1-\varkappa\Eb\big[I^{k}\big|\Fc_k\big]) (h(\hat{x}^k)-u^k)}\quad\\
    &\ge (1-\varkappa\Eb\big[I^{k}\big|\Fc_k\big]) (h(x^k)-u^k 
+ \langle \hat{x}^k-x^k,g_h^k\rangle-\frac{\delta}{2}\|\hat{x}^k-x^k\|^2)\\
    &\ge \langle \hat{x}^k-x^k,B^k\rangle-\frac{\delta}{2}\|\hat{x}^k-x^k\|^2
   + (1-\varkappa\Eb\big[I^{k}\big|\Fc_k\big]) (h(x^k)-u^k).  
\end{aligned}
\end{equation}

\noindent\mg{Combining \eqref{ineq-lower-bound-f-diff} and \eqref{ineq-lower-bound-f-diff-last-term}, we obtain}
\begin{align}
    f(\hat{x}^k,h(\hat{x}^k))-f(x^k,u^k) &\ge \langle \hat{x}^k-x^k,A^k+B^k\rangle-\frac{(1+\varkappa)\delta}{2}\|\hat{x}^k-x^k\|^2\nonumber \\
    &\qquad +(1-\varkappa\Eb\big[I^{k}\big|\Fc_k\big]) (h(x^k)-u^k). \label{lower_bound_f}
\end{align}
Now we consider the change in the Moreau envelope:
\begin{align*}
    \lefteqn{\Eb\big[\varphi_{1/\bar{\rho}}(x^{k+1})\,\big|\,\Fc_k\big]
    \le \Eb\big[F(\hat{x}^k) + \frac{\bar{\rho}}{2}\|\hat{x}^k-x^{k+1}\|^2\,\big|\,\Fc_k\big]}\quad\\
    &= F(\hat{x}^k) + \frac{\bar{\rho}}{2}\Eb\Big[\big\|\Pi_X\big( x^k - \tau \big( \tilde{g}_{fx}^k + \tilde{g}_{fu}^k \tilde{g}_h^k\big)^T\big) - \Pi_X(\hat{x}^k)\big\|^2\,\Big|\,\Fc_k\Big]\\
    &\le F(\hat{x}^k) + \frac{\bar{\rho}}{2}\Eb\Big[\big\| x^k- \hat{x}^k - \tau(\tilde{g}_{fx}^k + \tilde{g}_{fu}^k \tilde{g}_h^k)^T\big\|^2\,\Big|\,\Fc_k\Big]\\
    &\le F(\hat{x}^k) + \frac{\bar{\rho}}{2}\| x^k- \hat{x}^k\|^2 + \bar{\rho}\tau \Eb\big[\langle \hat{x}^k-x^k, \tilde{g}_{fx}^k + \tilde{g}_{fu}^k \tilde{g}_h^k\rangle\,\big|\,\Fc_k\big] +\frac{\bar{\rho}M^2}{2}\tau^2.
\end{align*}
Noticing that $\Eb[\tilde{g}_{fx}^k + \tilde{g}_{fu}^k \tilde{g}_h^k |\Fc_k] = A^k+B^k$ and plugging in the lower bound (\ref{lower_bound_f}), we obtain:
\begin{align*}
    \Eb[\varphi_{1/\bar{\rho}}(x^{k+1})|\Fc_k] &\le \varphi_{1/\bar{\rho}}(x^{k})+\frac{\bar{\rho}M^2}{2}\tau^2+\bar{\rho}\tau(f(\hat{x}^k,h(\hat{x}^k))-f(x^k,u^k)\\
    &\qquad +\frac{(1+\varkappa)\delta}{2}\|\hat{x}^k-x^k\|^2+(1-\varkappa\Eb\big[I^{k}\big|\Fc_k\big]) (u^k-h(x^k)))\\
    &\le \varphi_{1/\bar{\rho}}(x^{k})+\frac{\bar{\rho}M^2}{2}\tau^2+\bar{\rho}\tau(F(\hat{x}^k)-F(x^k)\\
    &\qquad +\frac{(1+\varkappa)\delta}{2}\|\hat{x}^k-x^k\|^2) + \bar{\rho}\tau\|u^k-h(x^k)\|\\
    &\qquad + \bar{\rho}\tau(F(x^k)-f(x^k,u^k)).\\
\end{align*}
Since the function $x \mapsto F(x)+\frac{\bar{\rho}}{2}\| x^k- x\|^2$ is strongly convex with the parameter $\bar{\rho}-\rho>0$, then
\begin{align*}
    F(x^k)-F(\hat{x}^k) &= (F(x^k)+\frac{\bar{\rho}}{2}\| x^k- x^k\|^2)-(F(\hat{x}^k)+\frac{\bar{\rho}}{2}\| x^k- \hat{x}^k\|^2)\\
    &\qquad +\frac{\bar{\rho}}{2}\| x^k- \hat{x}^k\|^2 \ge (\bar{\rho}-\frac{\rho}{2})\| x^k- \hat{x}^k\|^2.
\end{align*}
\mg{Recalling that} 
$\bar{\rho} = \rho+(1+\varkappa)\delta$, and using (\ref{me_grad}), we obtain
\begin{align*}
    \Eb[\varphi_{1/\bar{\rho}}(x^{k+1})|\Fc_k] &\le \varphi_{1/\bar{\rho}}(x^{k})+\frac{\bar{\rho}M^2}{2}\tau^2+\bar{\rho}\tau(-(\bar{\rho}-\frac{\rho}{2})\| x^k- \hat{x}^k\|^2\\
    &\qquad +\frac{(1+\varkappa)\delta}{2}\|\hat{x}^k-x^k\|^2) + \bar{\rho}\tau\|u^k-h(x^k)\|\\
    &\qquad + \bar{\rho}\tau(F(x^k)-f(x^k,u^k))\\
    &= \varphi_{1/\bar{\rho}}(x^{k})+\frac{\bar{\rho}M^2}{2}\tau^2-\frac{1}{2}\tau\|\nabla \varphi_{1/\bar{\rho}}(x^k)\|^2\\
    &\qquad + \bar{\rho}\tau\|u^k-h(x^k)\| + \bar{\rho}\tau(F(x^k)-f(x^k,u^k)).\numberthis\label{upper_bound_me1}
\end{align*}
Using the fact that $f(x,u)$ is 1-Lipschitz with respect to $u$ together with the tracking error bound (\ref{tracking_err}), we have
\[
    \Eb[F(x^k)-f(x^k,u^k)] \le \Eb[|h(x^k)-u^k|] \le C_2\tau^{1/2}\lz{+(1-\tau)^{k}|u^0-h(x^0)|},
\]
\lz{where $C_2 = \sigma+\sigma M$.} In (\ref{upper_bound_me1}), taking the expectation of both sides, we get
\begin{align*}
    \Eb[\varphi_{1/\bar{\rho}}(x^{k+1})] &\le \Eb[\varphi_{1/\bar{\rho}}(x^{k})] -\frac{1}{2}\tau\Eb[\|\nabla \varphi_{1/\bar{\rho}}(x^k)\|^2]+\frac{\bar{\rho}M^2}{2}\tau^2+2\bar{\rho}C_2\tau^{3/2}\nonumber \\
    &\qquad \lz{+2\bar{\rho}\tau(1-\tau)^{k}|u^0-h(x^0)|}.
\end{align*}
The summation over $k$ from $0$ to $N-1$ yields
\begin{align*}
    \Eb[\varphi_{1/\bar{\rho}}(x^{N})] &\le \varphi_{1/\bar{\rho}}(x^{0})-\frac{1}{2}\tau\sum_{k=0}^{N-1}\Eb[\|\nabla \varphi_{1/\bar{\rho}}(x^k)\|^2]\\
    &\qquad +NC_3\tau^{3/2}\lz{+2\bar{\rho}|u^0-h(x^0)|},
\end{align*}
\lz{where $C_3 = 2\bar{\rho}C_2$.} Lower bounding the left hand-side by $\lz{\min_{x\in X} F}$ and rearranging, we obtain the bound
\begin{align*}
    \Eb[\|\nabla \varphi_{1/\bar{\rho}}(x^R)\|^2]&=\frac{1}{N}\sum_{k=0}^{N-1}\Eb[\|\nabla \varphi_{1/\bar{\rho}}(x^k)\|^2\\
    &\le 2\frac{\varphi_{1/\bar{\rho}}(x^{0})-\lz{\min_{x\in X} F(x)}\lz{+2\bar{\rho}|u^0-h(x^0)|}+NC_3\tau^{3/2}}{N\tau},
\end{align*}
which was set out to prove. 
\end{proof}
\begin{remark}\label{remark-sample-complexity-smooth}
\mg{If we choose $\tau = cN^{-2/3}$ for some constant $c>0$, a consequence of Theorem \ref{smooth_rate} is that $
    \Eb[\|\nabla \varphi_{1/\bar{\rho}}(x^R)\|^2] =  \mathcal{O}(\frac{1}{N^{1/3}})$ and therefore in order to get an $\varepsilon$-optimal point, i.e. for $\Eb[\|\nabla \varphi_{1/\bar{\rho}}(x^R)\|^2] \leq \varepsilon$,
    we will need $\mathcal{O}(\varepsilon^{-3})$ iterations}. Since we use three data samples at each step in Algorithm \ref{alg1}, we obtain the total sample complexity of 
    $S_\varepsilon = 24(\varphi_{1/\bar{\rho}}(x^{0})-\lz{\min_{x\in X} F(x)}\\ \lz{ + 2\bar{\rho}|u^0-h(x^0)|}+C_3)^3 \varepsilon^{-3} = \mathcal{O}(\varepsilon^{-3})$.
\end{remark}

\section{Convergence rate for non-smooth weakly convex loss functions}
\label{s:4}

The SCS method uses a biased estimator of the
values of the inner function $h(x)$, which induces errors that are harder to bound when the loss function is non-smooth.
This is an organic difficulty in non-smooth and non-convex composite stochastic optimization, precluding the derivation of the convergence rate. Therefore, instead of updating the inner function estimate $\{u^k\}$ by a linear tracking filter,  we construct the SPIDER estimator \citep{fang2018spider}:
\lz{
\begin{equation}
\begin{aligned}
    u^{0} &= \ell_{\Bc^{0}}(x^{0}),\\
    u^{k} &= u^{k-1} + \ell_{\Bc^{k}}(x^{k}) - \ell_{\Bc^{k}}(x^{k-1}),
\end{aligned}
\label{def_spider}
\end{equation}
}
where $\Bc^{k}$ is a randomly picked mini-batch of data at the $k$-th iteration, and
$\ell_{\Bc^{k}}(x) \equiv \\ \frac{1}{|\Bc^{k}|}\sum_{D\in \Bc^{k}}\ell (x,D)$. \mg{This estimator operates in epochs and admits three parameters: the epoch length $T$, the standard batch size $b$, and the larger batch size $B>b$}. 
It restarts at the beginning of each epoch, uses the large batch at the first step, and the standard batch at the following steps. Furthermore, the SPIDER estimator is an unbiased estimator:
\begin{equation*}
    \Eb[u^k-h(x^k)] = 0,
\end{equation*}
where the expectation is taken with respect to all random observations up to iteration  $k$. The SCS method with SPIDER is described in Algorithm \ref{alg2}. 

\begin{algorithm}
\caption{\mg{SCS} method with SPIDER}\label{alg2}

\begin{algorithmic}[1]
\small
\Require initial point $x^0 \in X$, a constant stepsize $\tau \in \big(0, 1\big]$; SPIDER epoch length $T$, large batch size $B$ and small batch size $b$.
\For{$k=0,1,...,N-1$}
\lz{
\If{$k$ mod $T== 0$}
\State Randomly sample a data batch $\Bc^{k}$ with $|\Bc^{k}|==B$
\State Restart the inner function estimate:
\begin{equation*}
    u^{k} = \ell_{\Bc^{k}}(x^{k});
\end{equation*}
\Else
\State Randomly sample a data batch $\Bc^k$ with $|\Bc^k|==b$
\State Update the inner function estimate:
\begin{equation*}
    u^{k} = u^{k-1} + \ell_{\Bc^{k}}(x^{k}) - \ell_{\Bc^{k}}(x^{k-1}).
\end{equation*}
\EndIf
}
\State 
Randomly sample $D_1^{k+1}$ and $D_2^{k+1}$ 
and obtain the estimates
\begin{align*}
G^{k} &\in \partial_x \ell(x^{k}, D_1^{k+1}),\\
\tilde{g}_{fx}^{k} &= \begin{cases} 0 & \text{if } \ell(x^{k}, D_1^{k+1}) < u^{k},\\
                                    \varkappa G^{k} & \text{if } \ell(x^{k}, D_1^{k+1}) \ge u^{k},
                                    \end{cases}\\
\tilde{g}_{fu}^{k} &=\begin{cases} 1 & \text{if } \ell(x^{k}, D_1^{k+1}) < u^{k},\\
                                   1-\varkappa & \text{if } \ell(x^{k}, D_1^{k+1}) \ge u^{k},
                                    \end{cases}\\
\tilde{g}_{h}^{k} &\in \partial_x \ell(x^k, D_2^{k+1}). 
\end{align*}
\State Update the solution estimate
\begin{equation}
x^{k+1} = \Pi_X\Big( x^k - \mgtwo{\tau} \big( \tilde{g}_{fx}^k + \tilde{g}_{fu}^k \tilde{g}_h^k\big)^T\Big),\label{update_x_ns}
\end{equation}
\EndFor
\Ensure $x^R$ with $R$ uniformly sampled from $\{0,1,\dots,N-1\}$.
\end{algorithmic}
\end{algorithm}


\mg{For handling non-smooth losses, instead of Assumption \ref{as2}, we assume that the loss is only weakly convex:}
\begin{assumption}
\lz{For all $x$ in a neighborhood of the set $X$:
\begin{tightlist}{(ii)}
\item The function $\ell(x,\cdot)$ is integrable;
\item The function $\ell(\cdot,D)$ is weakly convex
with an integrable constant $\Tilde{\delta}(D)$.
\end{tightlist}
}
\label{as4}
\end{assumption}
\lz{\begin{remark}
Under Assumption \ref{as4}, the inner function $h(x)$ is $\delta$-weakly convex, where $\delta := \Eb[\Tilde{\delta}(D)]$. Since the outer function $f(x,u)$ is weakly convex with respect to $x$ and nondecreasing and convex with respect to $u$, the composite function $F(x)$ is also weakly convex.
\end{remark}}

By virtue of Assumptions \ref{as1} and \ref{as4}, the loss function $\ell (\cdot,D)$ (as a difference of a convex function and a quadratic function) is Lipschitz continuous on the feasible set $X$ for any arbitrary $D$, with some Lipschitz constant $\tilde{L}(D)$. We make an additional assumption
about this constant.

\begin{assumption}
The Lipschitz constant $\tilde{L}(D)$ of the loss function $\ell (x,D)$ with respect to $x$ is square-integrable: 
\[
L^2 \equiv \Eb[\tilde{L}^2(D)] < +\infty.
\]
\label{as5}
\end{assumption}

\begin{remark}
Assumption \ref{as5} implies the Mean-Squared Lipschitz (MSL) property \citep{nguyen2022finite,pham2020proxsarah} required by the SPIDER estimator:
\[
\Eb[|\ell(x,D)-\ell(y,D)|^2]\le L^2\|x-y\|^2.
\]
The composite subgradient bound \eqref{gF_bound} automatically follows in this case. 
\end{remark}

Now the loss function $\ell(x,D)$ becomes \lz{$\Tilde{\delta}(D)$}-weakly convex (instead of smooth) with respect to $x$ at any $D$, note that everything is still valid in the rate convergence analysis in Section 3, except the tracking error bound (\ref{tracking_err}). 
Therefore, we need to estimate the tracking errors of the SPIDER estimator; \mg{such estimates are already available in the literature as stated in the next result.}

\begin{lemma}
\cite[Lemma 1]{fang2018spider} Suppose the loss function $\ell(x,D)$ is Mean-Squared Lipschitz with a constant $L$. Then the MSE of the estimator in \eqref{def_spider} can be bounded as
\[
\Eb[|u^k-h(x^k)|^2] \le \Eb[|u^0-h(x^0)|^2] + \sum_{r=1}^{k}\frac{L^2}{|\Bc^r|}\Eb[\|x^r-x^{r-1}\|^2], \qquad k = 1,...,T-1.
\]
If we can control the step lengths $\|x^k-x^{k-1}\|\le s\,$ for $k=1,...,T-1$, then
\[
\Eb[|u^0-h(x^0)|^2]\le\cdots\le\Eb[|u^{T-1}-h(x^{T-1})|^2]\le\frac{\sigma^2}{B}+\frac{TL^2s^2}{b}.
\]
\label{spider_error}
\end{lemma}

\mg{Building on this lemma, we next obtain tracking error bounds for the  sequence $\{u^k\}$ for a particular choice of the SPIDER parameters $T$, $b$, and $B$.} 

\begin{lemma}
Suppose Assumption \ref{as1} and Assumptions \ref{as3} to \ref{as5} hold. Then for $B = \frac{2\sigma^2}{\tau^2}$, $b=2LM\sigma/\tau$, $T=\frac{\sigma}{LM\tau}$, the sequence $\{u^k\}$ generated by Algorithm \ref{alg2} satisfies:
\begin{equation}
\Eb\big[ |u^{k} - h(x^{k})| \big] \le  \tau, \qquad k = 0,1,...,N-1.
\label{tracking_err_ns}
\end{equation}
\label{nonsmooth_error}
\end{lemma}

\begin{proof}
According to the update function \eqref{update_x_ns} and the composite subgradient bound \eqref{gF_bound}, the step lengths are bounded:
\[
\|x^k-x^{k-1}\|\le M\tau, \qquad k = 0,1,...,N-1.
\]
Under Assumption \ref{as5}, $\ell(x,D)$ is Mean-Squared Lipschitz (MSL) with a constant $L$. 
Also, by virtue of Lemma \ref{spider_error}, the tracking errors of $\{u^k\}$ satisfy:
\[
\Eb[|u^k-h(x^k)|^2]\le \frac{\sigma^2}{B} + \frac{TL^2M^2\tau^2}{b},\qquad k=0,1,...,N-1.
\]
Therefore, if we choose
\[
B = \frac{2\sigma^2}{\tau^2},\qquad b=2LM\sigma/\tau,\qquad T=\frac{\sigma}{LM\tau},
\]
we can get an upper bound for the tracking errors:
\begin{align*}
    \Eb[|u^k-h(x^k)|^2]&\le \tau^2.
\end{align*}
Using Jensen's inequality, we immediately get \eqref{tracking_err_ns}.
\end{proof}

The remaining rate analysis follows the same way as in the continuously differentiable case and we obtain the following result. 

\begin{theorem}
\label{nonsmooth_rate}
If Assumption \ref{as1} and Assumptions \ref{as3} to \ref{as5} hold,then for every $N\ge 1$, the random point $x^R$ generated by Algorithm \ref{alg2} satisfies:
\begin{equation*}
    \Eb[\|\nabla \varphi_{1/\bar{\rho}}(x^R)\|^2]
    \le 2\frac{\varphi_{1/\bar{\rho}}(x^{0})-\lz{\min_{x\in X} F(x)}+NC_4\tau^2}{N\tau},
\end{equation*}
where \lz{$C_4 = 2\bar{\rho}(\frac{1}{4}M^2+1)$}, the expectation is taken with respect to the trajectory generated by Algorithm \ref{alg2}, and the random variable $R$ is uniformly sampled from $\{0,1,\dots,T-1\}$ and independent of the trajectory.

\end{theorem}
\begin{proof}
The proof follows in the same way as Theorem \ref{smooth_rate} \mg{with the exception that the tracking bound (\ref{tracking_err}) is to be replaced with the bound \eqref{tracking_err_ns}}.
\end{proof}

\begin{remark}
If we choose $\tau = N^{-1/2}$, in order to get an $\varepsilon$-optimal point, we will need $\mathcal{O}(\varepsilon^{-2})$ iterations. 
With the choices of hyperparameters in Lemma \ref{nonsmooth_error}, the average batch size per iteration will be
\[
b + B/T = 4LM\sigma/\tau,
\]
so the \mg{total} sample complexity we obtain is $$S_\varepsilon = 32(\varphi_{1/\bar{\rho}}(x^{0})-\lz{\min_{x\in X} F(x)}\\+C_4)^3LM\sigma \varepsilon^{-3} =\mathcal{O}(\varepsilon^{-3}). $$ 
\mg{We conclude that the sample size of $O(\varepsilon^{-3})$ is sufficient for both smooth and non-smooth cases, with only the constant larger in the non-smooth case (see Remark \ref{remark-sample-complexity-smooth}).}
\end{remark}

\section{Numerical Experiments}

In this section, we report results of numerical experiments that \mg{illustrate the convergence and robustness of our SCS method on an adversarial learning task in deep learning and on some logistic regression problems with non-smooth non-convex regularizers}.
{Our numerical results were obtained using Python (Version 3.7) on an Alienware Aurora R8 desktop with a 3.60 GHz CPU (i7-2677M) and 16GB memory.}

\subsection{Deep Learning}

We consider a convolutional neural network applied to the MNIST data set \citep{lecun2010mnist}. The network consists of three convolutional layers followed by a dense layer. All the hidden layers have ELU activations, and the output layer has the softmax activation. The convolutional layers have 16, 32, and 32 kernels of size 8, 6, and 5, respectively. We construct the network in this particular way to generate comparable results to \cite{sinha2018certifying}.

{The MNIST dataset consists of \mg{handwritten images with an integer label valued from 0 to 9 where the aim 
is to} classify the images. We use the cross entropy loss during training 
(see reference \href{https://pytorch.org/docs/stable/generated/torch.nn.CrossEntropyLoss.html}{here}). The resulting loss function $\ell(x,D)$ is a composition of the CNN and the cross-entropy loss, and is continuously differentiable in our setting, \mg{as the ELU activation functions are continuously differentiable}. \mgtwo{We also set a bound on the weights of all hidden units, which is equivalent to choosing the feasibility set $X = \{x : \|x\|_\infty \leq 10\}$. Such box constraints are employed frequently in practice for regularization purposes \citep{srivastava2014dropout}}.}

We train the CNN with different optimizers, namely SGD, SCS (with different values of $\varkappa$ in Algorithm \ref{alg1}) and another state-of-the-art method Wasserstein Robust Method (WRM) 
\citep{sinha2018certifying}.
\lz{To investigate the robustness of the trained networks, we \mgtwo{consider} 
two types of \mgtwo{(adversarial attacks) perturbations} to the test dataset}: the PGM attacks \citep{sinha2018certifying,mkadry2017towards} and the semi-deviation attacks, \mgtwo{which we will describe next}. 
\paragraph{PGM attack.}
\mgtwo{Given model parameter} $x$ \mgtwo{and data point $D = (a,b)$ with input $a$ and output $b$, the main idea is to create an adversarial input data by applying multi-step projected gradient ascent to the loss function in a ball around the data point.} 
Specifically, for every data point $(a_i,b_i)$, we iterate
\begin{align*}
\nabla a_i^{t}(x) &:= \argmax_{\|\eta\|_2\le \epsadv}\{ \nabla_a\ell(x;a_i^t,b_i)^T\eta \},\\
a_i^{t+1} &:= \Pi_{B(a^t_i)}\{ a^t_i + \tau_{\textit{adv}}\nabla a_i^t(x) \},
\end{align*}
for $t=1,...,T_\textit{adv}$, where $B(a^t_i):=\{a: \|a-a_i^t\|_2\le\epsadv\}$ is a ball around $a^t_i$, $\epsadv$ controls the adversarial perturbation level and \mgtwo{$T_\textit{adv}$ is the number of iterations. We refer the reader to \citep{mkadry2017towards} for further details.}
\paragraph{Semi-deviation attack.}
\mg{While PGM attacks create a powerful worst-case adversary, this can be pretty conservative if perturbations to the dataset has a random-like nature rather than a worst-case nature. We introduce semi-deviation attacks which create an adversary that is not as conservative by replacing} 
a ``good'' instance with an ``average'' instance. More specifically, for every data point $D_i$, we replace the corresponding loss $\ell(x,D_i)$ with $\ell_\textit{sd}(x,D_i)$:
\begin{align*}
    \bar{\ell}(x) &:= \frac{1}{|\Bc_\textit{test}|}\sum_{D_j\in \Bc_\textit{test}}\ell(x,D_j),\\
    \ell_\textit{sd}(x,D_i) &:=  \bar{\ell}(x) + \varkappa_\textit{adv} \max(0, \ell(x,D_i)-\bar{\ell}(x)),
\end{align*}
where $\Bc_\textit{test}$ denotes the test data set, $\varkappa_\textit{adv}$ controls the adversarial perturbation level. 



We control the number of gradient evaluations in training to be the same across different methods. \footnote{
\lz{
We train the model with SCS for 20 epochs and with WRM for 4 epochs, since at each step, SCS evaluates 3 gradients and WRM evaluates 16 gradients.
}
} 
\mgtwo{The training data is the original (uncontaminated) MNIST data, whereas the models are tested with the contaminated data subject to PGM attacks and semi-deviation attacks:} \lz{for each data point in the test set, we apply the PGM attack and the semi-deviation attacks} at different perturbation levels and plot the average test losses of different models at different perturbation levels in Plots (a) and (c) of Figure \ref{exp_ml}. When $\epsadv = 0.6$ for PGM attacks, or $\varkappa_\textit{adv} = 1$ for semi-deviation attacks, \lz{we also calculate the 
 (natural) logarithm of the losses in the test set and show the logarithm of the loss distributions of different models} in Plots (b) and (d) of Figure \ref{exp_ml}. We see that \mgtwo{SCS} with a proper $\varkappa$ value generates a better solution than SGD under both types of attacks. It's not as good as WRM under the PGM attacks, which was expected, since WRM is trained against very similar attacks.\lz{\footnote{
At the $k$-th step, WRM perturbs the data point $D^k = (a^k,b^k)$ with gradient ascent \mgtwo{applied to the cost} $a\mapsto \ell(x;a,b^k)-\gamma_{\textit{adv}}c(a,a^k)$, where $\gamma_{\textit{adv}}$ \mgtwo{is a parameter that} controls the robustness and $c(a,a^k) := \|a-a^k\|^2$ \mgtwo{is a regularizer}. 
In the experiment, we iterate 15 times during the gradient ascent at each step.
}
}For the same reason, SCS outperforms WRM under semi-deviation attacks.


\begin{figure}[ht]
\centering
\begin{subfigure}{0.48\textwidth}
  \includegraphics[width=1\linewidth]{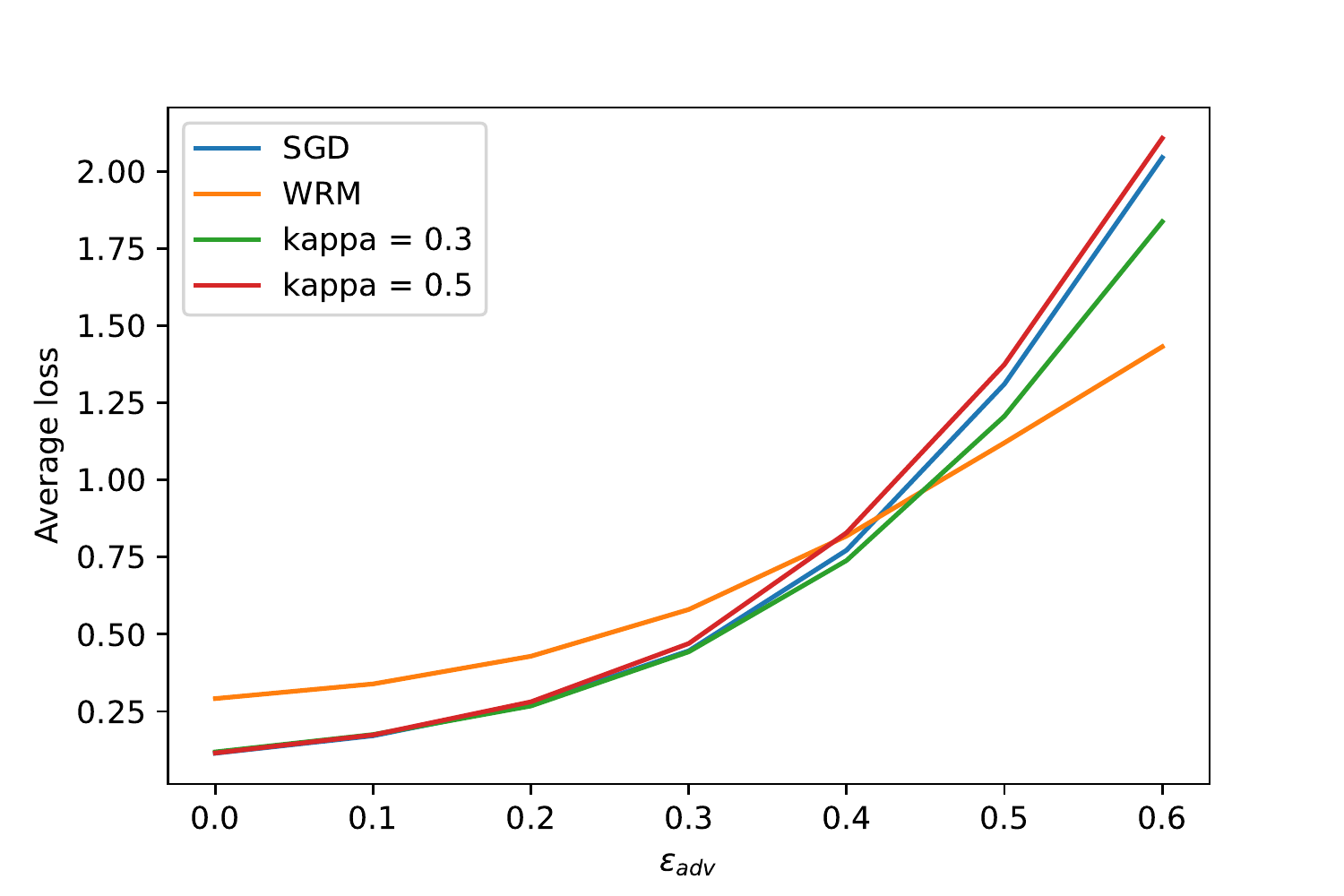}
  \caption{Average losses under different perturbation\\ levels of PGM attacks}
\end{subfigure}
\begin{subfigure}{0.48\textwidth}
  \includegraphics[width=1\linewidth]{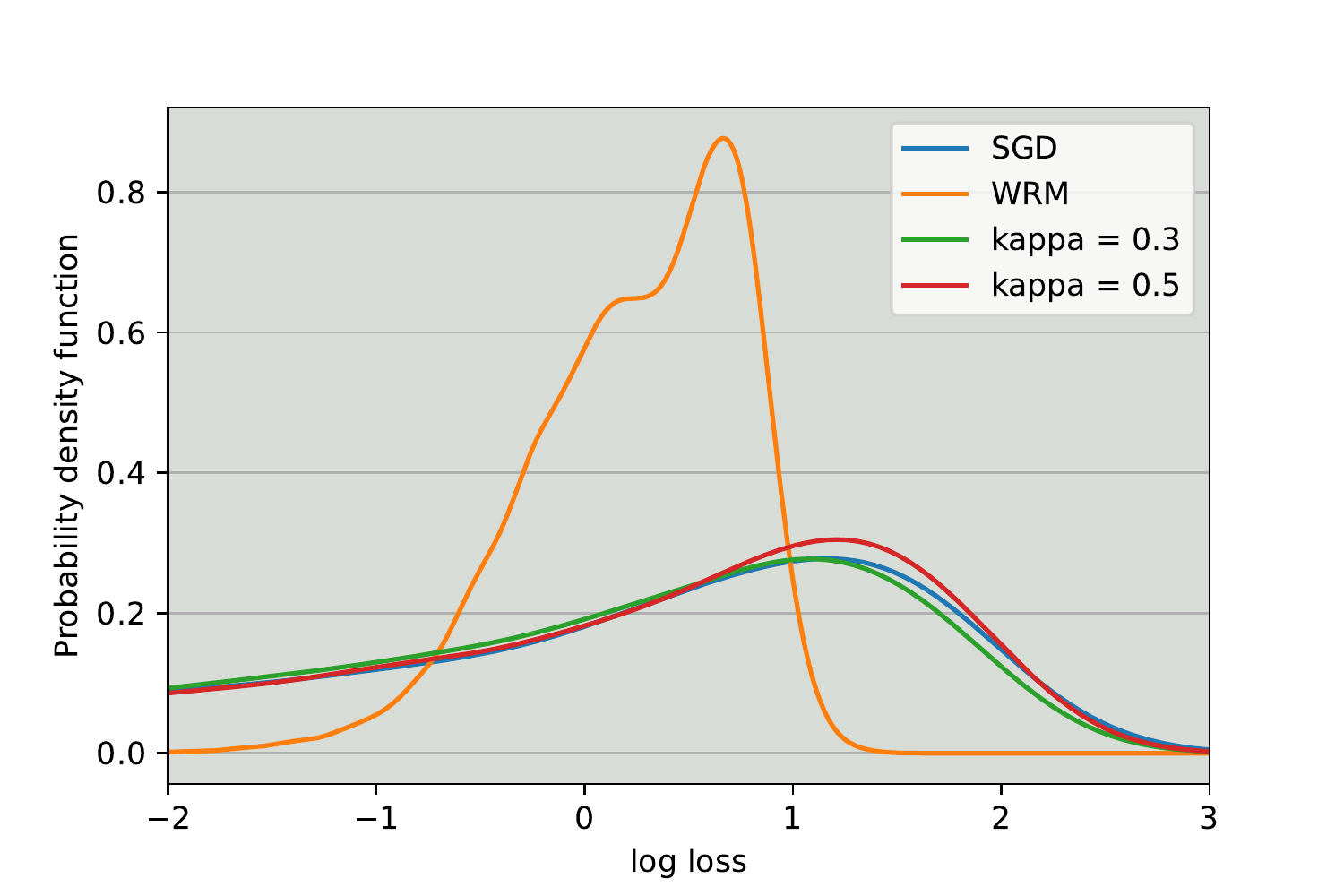}
  \caption{\mgtwo{Probability density function (PDF) of the logarithm of the loss function subject to} PGM attacks}
\end{subfigure}
\begin{subfigure}{0.48\textwidth}
  \includegraphics[width=1\linewidth]{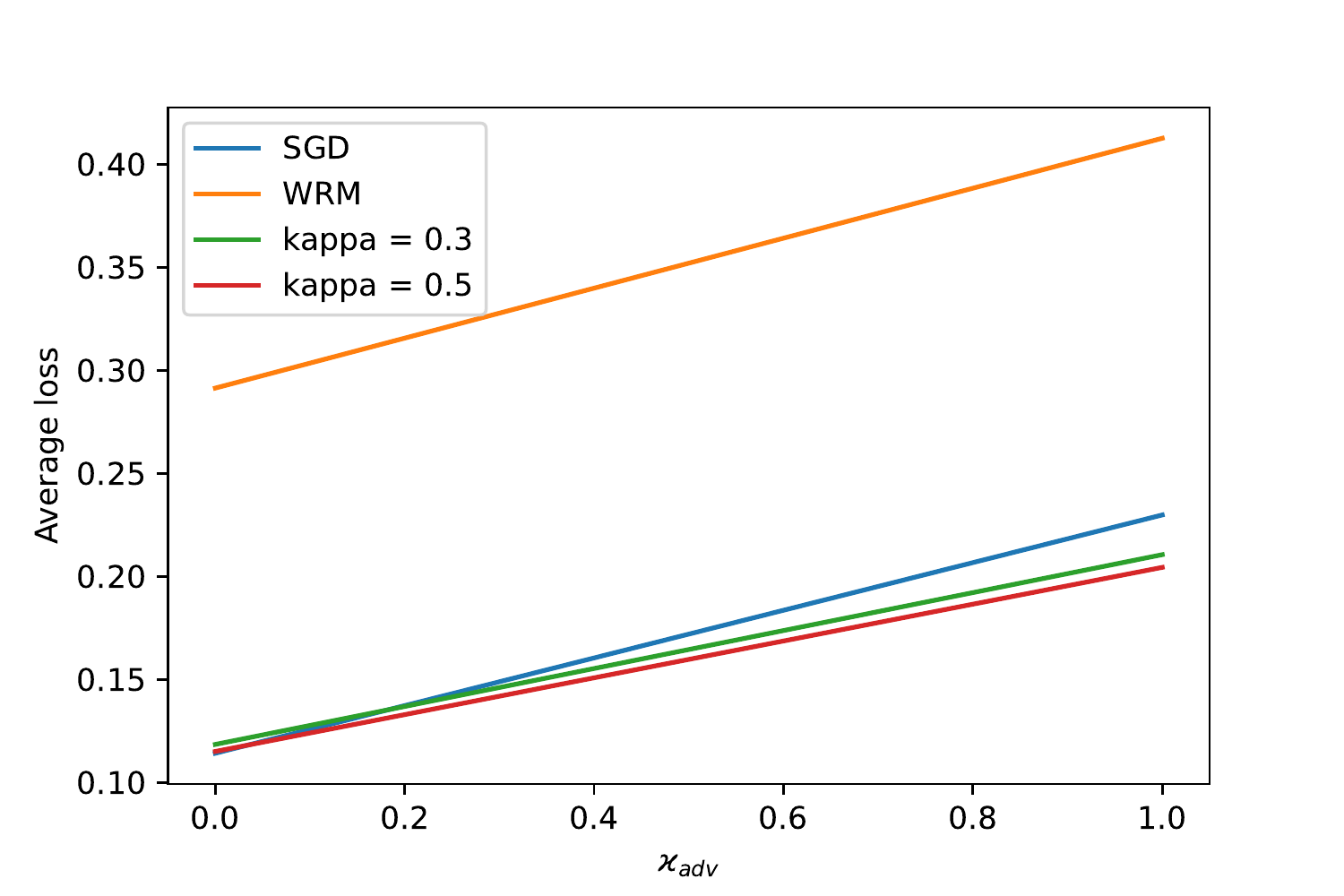}
  \caption{Average losses under different perturbation\\ levels of semi-deviation attacks}
\end{subfigure}
\begin{subfigure}{0.48\textwidth}
  \includegraphics[width=1\linewidth]{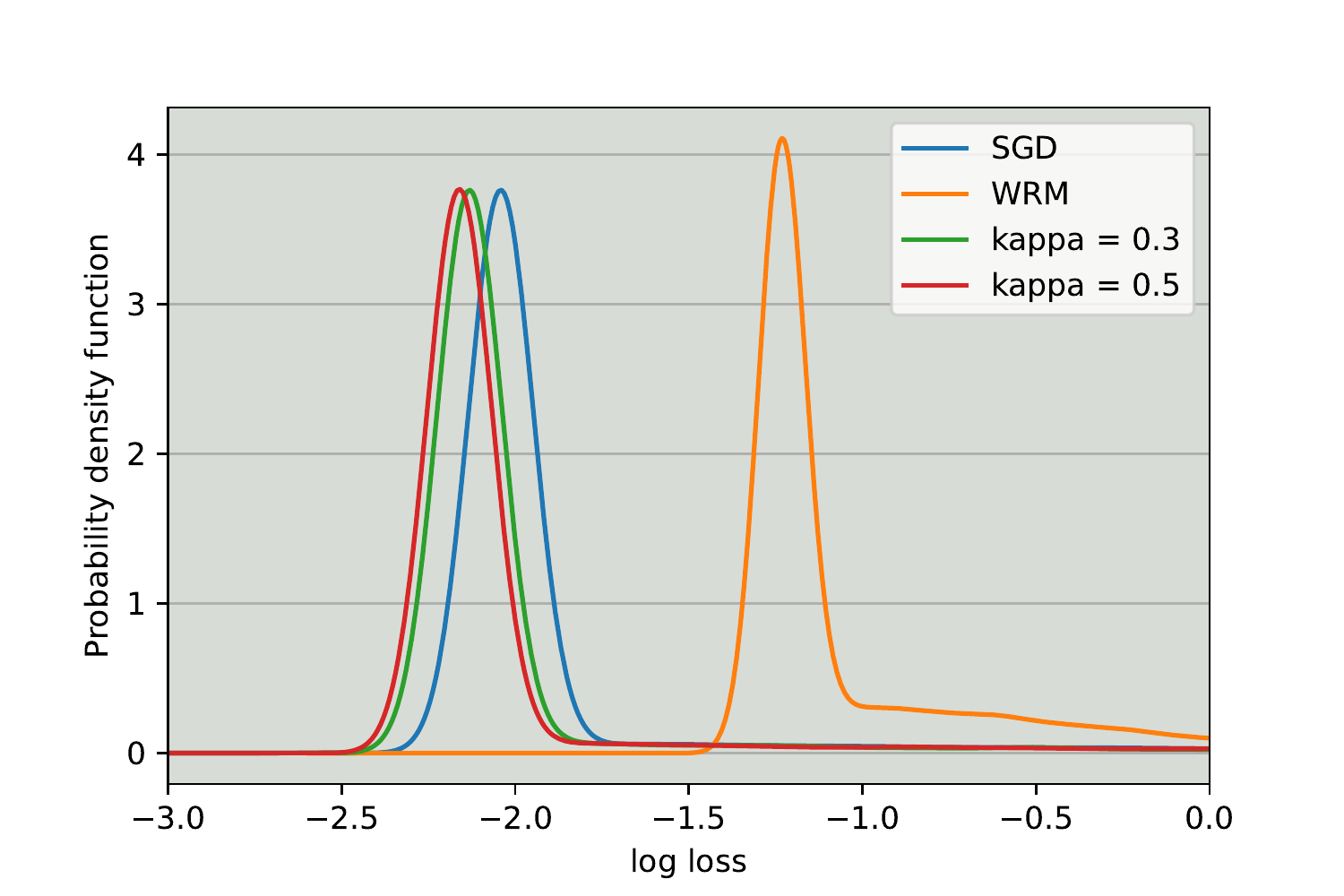}
  \caption{\mgtwo{PDF of the logarithm of the loss function subject to} semi-deviation attacks}
\end{subfigure}
\caption{Test losses under PGM attacks and semi-deviation attacks. The training data is the original (uncontaminated) MNIST data, whereas the models are tested with the contaminated data.}
\label{exp_ml}
\end{figure}

\subsection{Nonconvex penalties}

We consider a regression task on the Blog Feedback data set \citep{buza2014feedback}, containing 281 variables 
extracted from blog posts. The task is to predict the number of comments based on the other 280 (variables)
features. The instances in the years 2010 and 2011 are included in the training set (52396 in total), the instances on $02/01/2012$ are included in the validation set (114 in total), and the instances between $02/02/2012$ and $03/31/2012$ are included in the test set (7511 in total). The test set is divided into 60 subsets, each containing instances generated in one day from February or March. We use linear regression with mean absolute difference (MAD) loss as our model, plus a regularization term. The loss function has the form $\ell (x,D) = |a^Tx-b| + r(x)$ where $D=(a,b)$ is the input data, and $r(x)$ is the regularization term. For different choices of regularization terms,
the loss function can be convex or nonconvex.
Here we experiment on the Lasso penalty \citep{frank1993statistical} and two non-convex penalties: the SCAD penalty \citep{fan2001variable} and the MCP penalty \citep{zhang2010nearly}. The corresponding regularizers are:
\begin{itemize}
    \item Lasso:
    \begin{equation}
        r(x) = \lambda |x|,\label{Lasso}
    \end{equation}
    \item SCAD:
    \begin{equation}
        r(x) = \begin{cases}
        \lambda |x| & \text{if $|x| \le \lambda$,}\\
        \frac{\gamma\lambda |x| - 0.5(x^2+\lambda^2)}{\gamma-1} & \text{if $\lambda< |x| \le \lambda\gamma$,}\\
        \frac{\lambda^2(\gamma+1)}{2} & \text{if $|x| > \lambda\gamma$\,,}
        \end{cases}
    \label{SCAD}
    \end{equation}
    \item MCP:
    \begin{equation}
        r(x) = \begin{cases}
        \lambda |x| - \frac{x^2}{2\gamma} & \text{if $|x| \le \lambda\gamma$,}\\
        \frac{\lambda^2\gamma}{2} & \text{if $|x| > \lambda\gamma$,}
        \end{cases}
    \label{MCP}
    \end{equation}
\end{itemize}
\mg{where $\lambda>0$, and $\gamma>0$ are parameters.} With SCAD or MCP penalties, the loss function becomes non-smooth and weakly convex. \mgtwo{We take the constraint set to be $X = \{x : \|x\|_\infty \leq 10\}$}.


\mgtwo{In Figure \ref{exp_reg}, we compare the histories of training losses and the distributions of the logarithm of the test losses for these three different penalties together with a plot of the decay of the objective function during the training phase. }
The method exhibits similar convergence speed and test performance in (Lasso) convex and (SCAD and MCP) nonconvex cases. 
\begin{figure}[ht]
\centering
\begin{subfigure}{0.49\textwidth}
  \includegraphics[width=1\linewidth]{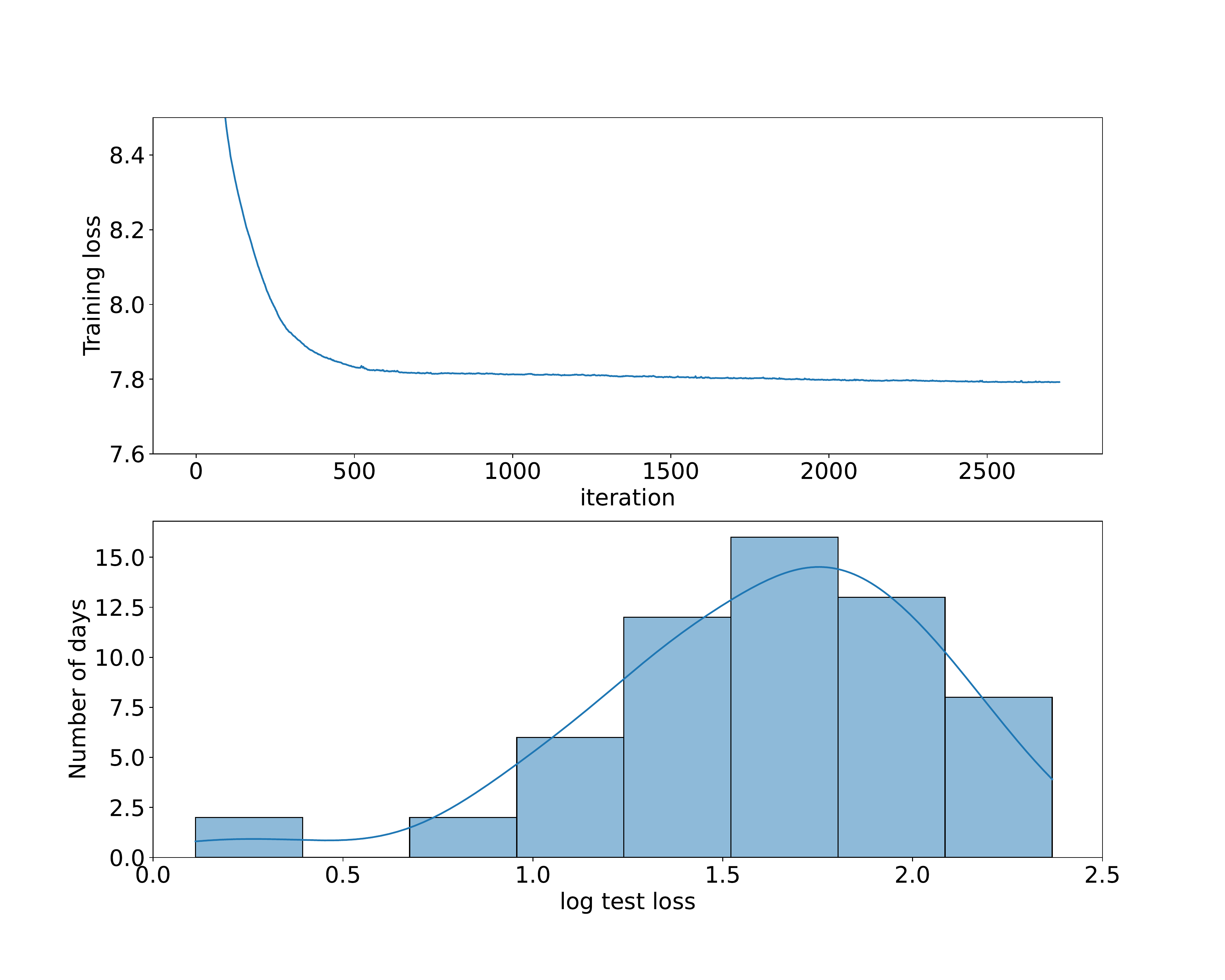}
  \caption{Lasso penalty}
\end{subfigure}
\begin{subfigure}{0.49\textwidth}
  \includegraphics[width=1\linewidth]{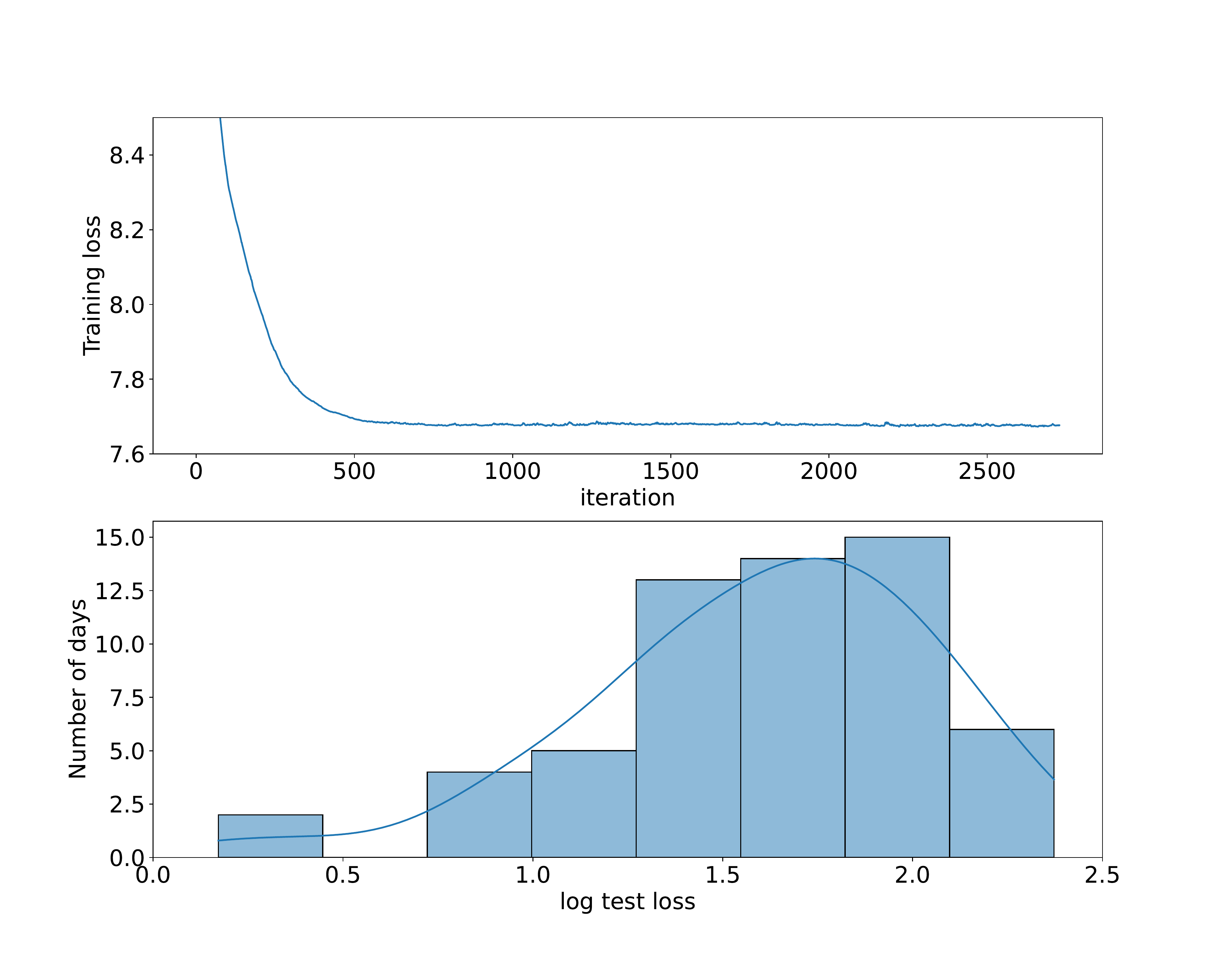}
  \caption{SCAD penalty}
\end{subfigure}
\begin{subfigure}{0.49\textwidth}
  \includegraphics[width=1\linewidth]{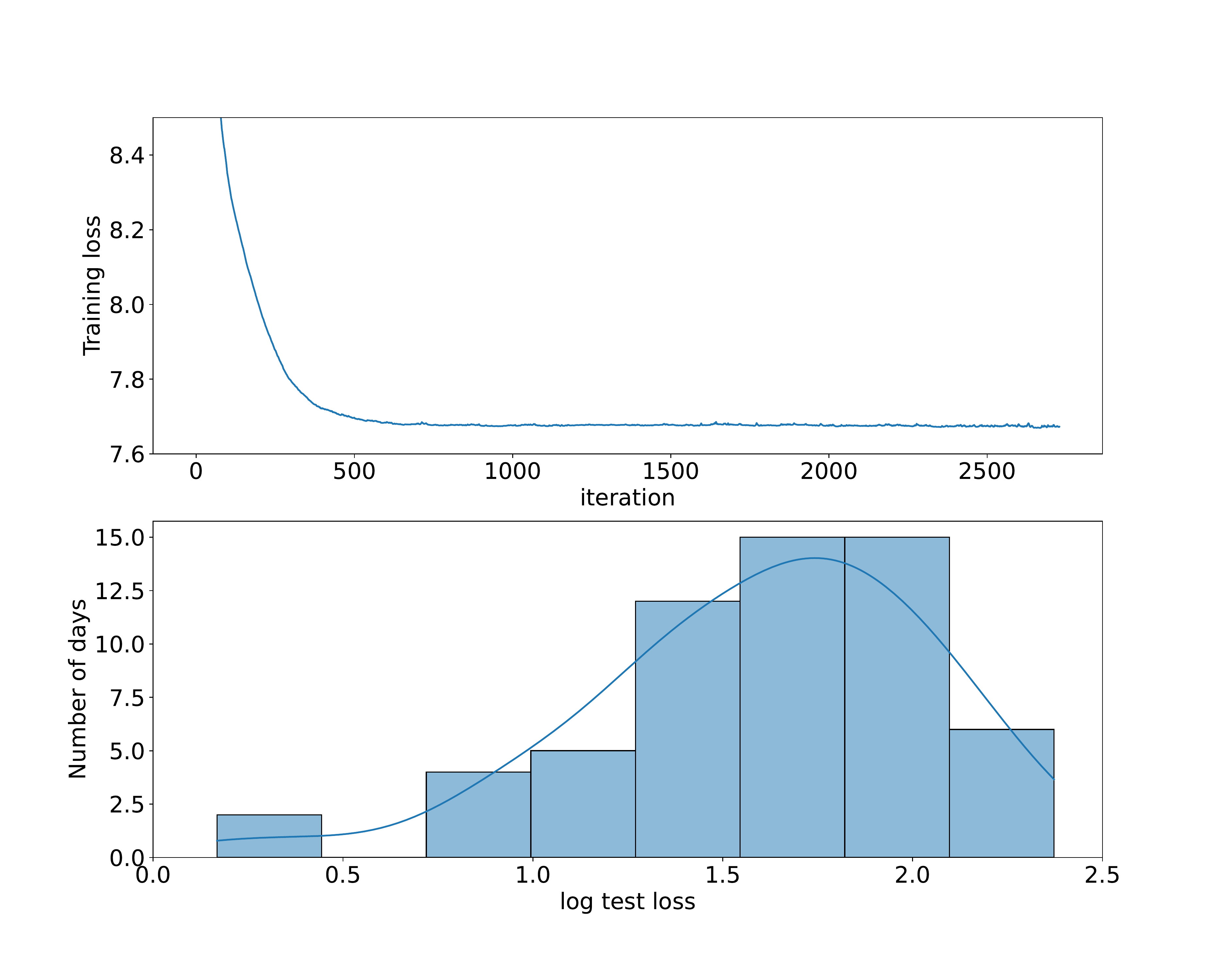}
  \caption{MCP penalty}
\end{subfigure}
\caption{For Lasso, SCAD and MCP penalties, the top image in each plot shows the training loss along iterations, and the bottom image in each plot shows the \mgtwo{logarithm of the distribution of the loss} on the test data.}
\label{exp_reg}
\end{figure}

\subsection{Remarks on the assumptions}\label{subsec-assump-hold}
For our deep learning experiment, the feasibility set $X = \{x : \|x\|_\infty \leq 10\}$ is clearly convex and compact, and Assumption \ref{as1} holds.
The input data are normalized and bounded, and the iterates stay in the compact set $X$ where the gradients of the continuously differentiable loss \mgtwo{$\ell(x,D)$} is continuous and bounded, so the loss $\ell(x,D)$ is integrable with respect to $D$ for every fixed $x\in X$. Therefore, Assumption \ref{as2} holds. 
Furthermore, we observe from \eqref{def_G} to \eqref{def_h} that the sequences $\tilde{J}^{\,k}$, $\tilde{g}^{k}$, and $\tilde{h}^{k}$ stay uniformly bounded. Therefore, their variance (conditioned on the natural filtration $\mathcal{F}_k$) is bounded, and the stochastic estimate of \mgtwo{an element from the subdifferential} $\partial F(x^k)$ is also bounded. 
If we take the expectation of these estimates, as the subdifferentials are bounded sets, we can interchange the subdifferential and the expectation operators \cite[Thm. 23.1]{mikhalevich1987nonconvex}. Since $D_1^{k}$, $D_2^k$ and $D_3^{k}$ are i.i.d. samples from the empirical data distribution, then
we can deduce that $\tilde{J}^{\,k}$, $\tilde{g}^{k}$ and $\tilde{h}^{k}$ are unbiased estimates. From these observations, we conclude that Assumption \ref{as3} is also satisfied.

For the regularized \mgtwo{logistic} regression example, similar to \citep{mei2018landscape}, the constraint set $X$ 
is an $\ell_\infty$ ball so that it is convex and compact and Assumption \ref{as1} holds. \mgtwo{All the penalty functions we considered in \eqref{Lasso}--\eqref{MCP} are weakly convex, so that the loss $\ell(x,D)$ is also weakly convex with respect to $x$}. \mgtwo{By similar arguments to those for the deep learning setting}, the loss $\ell(x,D)$ is integrable with respect to $D$ for every fixed $x\in X$, and Assumptions \ref{as3} and \ref{as4} hold. The input data are normalized and bounded, so the stochastic Lipschitz constant $\tilde{L}(D)$ is also bounded, and Assumption \ref{as5} holds. \mgtwo{Therefore, our assumptions are satisfied for the numerical experiments conducted in this work.}

\section{Conclusion}

\mgtwo{In this paper, we considered a distributionally robust stochastic optimization problem where the ambiguity set is defined with the use of the mean--semideviation risk measure. We reformulated this problem as a stochastic two-level non-smooth optimization problem and proposed a single time-scale method called Stochastic Compositional Subgradient (SCS). Our method can support two different ways of inner value tracking: (i) linearized tracking of a continuously differentiable loss function, (ii) tracking of a weakly convex \mg{loss} function through the SPIDER estimator. We show that the sample complexity of $\mathcal{O}(\varepsilon^{-3})$ is possible in both cases, with only the constant larger in the second case. To our knowledge, this is the first sample complexity result for distributionally robust learning with non-convex non-smooth losses. Finally, we illustrated the performance of our algorithm on a robust deep-learning problem and a logistic regression problem with weakly convex, non-smooth regularizers.}

\section{Acknowledgements}
This research is supported in part by the grants Office of Naval Research Award Number N00014-21-1-2244, National Science Foundation (NSF) CCF-1814888, NSF DMS-2053485, NSF DMS-1723085, the National Science Foundation Award DMS-1907522 and
by the Office of Naval Research Award N00014-21-1-2161.

\bibliography{multi}


\end{document}